\documentclass{amsart}

\usepackage{amsmath}
\usepackage{amsthm}
\usepackage{amssymb}
\usepackage{amscd}
\usepackage{amsfonts}
\usepackage{amsbsy}
\usepackage[noadjust]{cite} 
\usepackage{paralist}
\usepackage{psfrag}
\usepackage{mathrsfs}       
\usepackage{verbatim}

\theoremstyle{plain}
\newtheorem {theorem} {Theorem} [section]

\newtheorem {lemma} [theorem] {Lemma}

\theoremstyle{definition}

\newtheorem {remark}[theorem]{Remark}
\newtheorem {example}[theorem] {Example}
\newtheorem {algorithm}[theorem] {Algorithm}

\newcommand{\N}{\mathbb{N}}

\newcommand{\R}{\mathbb{R}}

\newcommand{\Z}{\mathbb{Z}}

\newcommand{\cc}{\mathscr C}

\newcommand{\F}{\mathcal F}
\newcommand{\T}{{\mathscr T}}

\newcommand{\Fb}{{\bf F}}

\newcommand{\yb}{{\bf y}}

\newcommand{\eb}{\pmb{\eta}}
\def \g {\gamma}

\newcommand{\mb}{\pmb{\mu}}

\begin{document}

\title[Boundary Value Problems]
{Symbolic Iterative Solution of Two-Point Boundary Value Problems}

\author[H.~Semiyari \& D.~S.~Shafer]{Hamid Semiyari$^1$ and Douglas S. Shafer$^2$}

\address{\noindent $^1$ Mathematics Department, James Madison University, Harrisonburg, Virgina 22807, USA}

\email{msah.sem@gmail.com}

\address{\noindent $^2$ Mathematics Department, University of North Carolina at Charlotte, Charlotte,
North Carolina 28223, USA}

\email{dsshafer@uncc.edu}



\begin{abstract}
In this work we give an efficient method involving symbolic manipulation, Picard iteration, and auxiliary variables
for approximating solutions of two-point boundary value problems.
\end{abstract}

\maketitle

\section{Introduction}\label{s:intro}

There exist a variety of numerical methods for approximating solutions of two-point boundary value problems,
among them shooting methods, finite difference techniques, power series methods, and variational methods, all
of which are described in detail in classical texts (for example, \cite{BSW}, \cite{BF}, and \cite{K}). With the
advent of computer algebra systems hybrids of numerical and symbolic manipulation techniques have also arisen 
(for example, \cite{PP} and \cite{SW}). In this work we develop a purely symbolic technique for approximating
solutions of two-point boundary value problems that applies identically in both linear and nonlinear cases.

Fundamental to the technique is the idea of deriving an integral expression for the slope $\g$ of the solution 
$y(t)$ of the boundary value problem at the left endpoint; we then use a Picard iteration scheme to simultaneously
approximate, ever more closely, both $\g$ and $y(t)$, the latter now viewed as the unique solution to the initial
value problem that it determines at the left endpoint. We prove theorems guaranteeing both existence of a unique
solution to the boundary value problem and convergence of our iterates to it, although as we demonstrate the
technique works under conditions far more general than those given by the theorems. Since the theorems are proved
using the Contraction Mapping Theorem, the iterates obtained converge to $y(t)$ exponentially fast in the supremum
norm.

By introducing auxiliary variables we overcome problems with quadratures that cannot be performed in closed
form. Thus we ultimately obtain an efficient computational method whose output is a sequence of polynomials
converging to $y(t)$.

\section{The Algorithm}\label{s:alg}

Consider a two-point boundary value problem of the form
\renewcommand{\thefootnote}{}
\footnote{${}^1$ Mathematics Department, James Madison University, Harrisonburg, VA 22807, USA}
\footnote{${}^2$ Mathematics Department, University of North Carolina at Charlotte, Charlotte, NC 28223, USA 
                 (Corresponding author, \texttt{dsshafer@uncc.edu}, +1 704 687 5601, FAX +1 704 687 1392)}
\begin{equation}\label{basic.bvp}
y'' = f(t, y, y'), \qquad y(a) = \alpha, \quad y(b) = \beta \, .
\end{equation}
If $f$ is continuous and locally Lipschitz in the last two variables then by the Picard-Lindel\"of Theorem, for any 
$\gamma \in \R$ the initial value problem
\begin{equation}\label{basic.ivp}
y'' = f(t, y, y'), \qquad y(a) = \alpha, \quad y'(a) = \gamma
\end{equation}
will have a unique solution on some interval about $t = a$. 
The boundary value problem \eqref{basic.bvp} will have a solution if and only if there exists $\gamma \in \R$ such
that (i) the maximal interval of existence of the unique solution of \eqref{basic.ivp} contains the interval 
$[a, b]$, and (ii) the unique solution $y(t)$ of \eqref{basic.ivp} satisfies $y(b) = \beta$. But \eqref{basic.ivp} is
equivalent to the integral equation
\begin{equation}\label{equiv.ivp}
y(t) = \alpha + \gamma (t - a) + \int_a^t (t - s) f(s, y(s), y'(s)) \, ds.
\end{equation}
If $y(t)$ is a solution of \eqref{basic.bvp} then inserting it into \eqref{equiv.ivp}, evaluating at $t = b$, and 
solving the resulting equation for $\gamma$, we obtain an expression for the corresponding value of $\gamma$ in \eqref{basic.ivp}, namely
\begin{equation} \label{e:gamma}
\gamma = \tfrac1{b - a}
                \left(
                \beta - \alpha - \int_a^b (b - s) f(s, y(s), y'(s)) \, ds
                \right).
\end{equation}
(If \eqref{basic.bvp} has no solution then for any solution $y(t)$ of the ordinary differential equation in \eqref{basic.bvp} that
satisfies $y(a) = \alpha$ and exists on $[a, b]$ the number on the right hand side of \eqref{equiv.ivp}, hence the value of $\gamma$ specified by \eqref{e:gamma}, exists, but will not be equal to the originally determined value of $y'(a)$.) The key idea in the new
method proposed here for solving \eqref{basic.bvp} is that in a Picard iteration scheme applied to the system of first order equations
\begin{align*}
y' &= u                  \mspace{124mu}y(a) = \alpha \\
u' &= f(t, y(t), u(t))    \mspace{25mu}u(a) = \gamma
\end{align*}
that is equivalent to \eqref{basic.ivp} we use \eqref{e:gamma} to iteratively obtain successive approximations to the value of $\gamma$
in \eqref{basic.ivp}, if it exists. Thus making for the initial approximations of $y(t)$ and $u(t) = y'(t)$, the reasonable respective
choices of the left boundary value and the average slope of the solution to \eqref{basic.bvp} on the interval $[a, b]$, the iterates are
\begin{subequations} \label{e:iteration.split}
\begin{equation} \label{e:iteration.split.initial}
\begin{aligned}
y^{[0]}(t)   &\equiv \alpha \\
u^{[0]}(t)   &\equiv \tfrac{\beta - \alpha}{b - a} \\
\end{aligned}
\end{equation}
and 
\begin{equation}\label{e:iteration.split.recursion}
\begin{aligned}
\gamma^{[k+1]} &= \frac1{b-a}
                  \Big(
                  \beta - \alpha - \int_a^b (b - s) f(s, y^{[k]}(s), u^{[k]}(s)) \, ds
                  \Big) \\
y^{[k+1]}(t)   &= \alpha + \int_a^t u^{[k]}(s) \, ds \\
u^{[k+1]}(t)   &= \gamma^{[k+1]} + \int_a^t f(s, y^{[k]}(s), u^{[k]}(s)) \, ds \,.
\end{aligned}
\end{equation}
\end{subequations}
This gives the following algorithm for approximating solutions of \eqref{basic.bvp}.

\begin{algorithm} \label{r:alg1}
To approximate the solution of the boundary value problem
\begin{equation}\label{basic.bvp.dupe}
y'' = f(t, y, y'), \qquad y(a) = \alpha, \quad y(b) = \beta 
\end{equation}
iteratively compute the sequence of functions on $[a, b]$ defined by \eqref{e:iteration.split.initial}
and \eqref{e:iteration.split.recursion}.
\end{algorithm}

We will now state and prove a theorem that gives conditions guaranteeing that the problem \eqref{basic.bvp.dupe}
has a unique solution, then prove that the iterates in Algorithm \ref{r:alg1} converge to it. We will need the
following simple lemma whose proof is omitted.

\begin{lemma}\label{lemma.L}
Let $E \subset \R \times \R^2$ be open and let $f: E \to \R : (t, y, u) \mapsto f(t, y, u)$ be Lipschitz in
$\yb = (y,u)$ on $E$ with Lipschitz constant $L$ with respect to absolute value on $\R$ and the sum norm on
$\R^2$. Then 
\[
\Fb : E \to \R^2 : (t, y, u) \mapsto (u, f(t, y, u))
\]
is Lipschitz in $\yb$ with Lipschitz constant $1 + L$ with respect to the sum norm on $\R^2$.
\end{lemma}

\begin{theorem}\label{thm.1}
Let $f : [a, b] \times \R^2 \to \R : (t, y,u) \mapsto f(t, y, u)$ be continuous and Lipschitz in $\yb = (y, u)$
with Lipschitz constant $L$ with respect to absolute value on $\R$ and the sum norm on $\R^2$. If 
$0 < b - a < (1 + \frac32 L)^{-1}$ then for any $\alpha$, $\beta \in \R$ the boundary value problem
\begin{equation} \label{orig_bvp.1}
y'' = f(t, y, y'), \qquad y(a) = \alpha, \quad y(b) = \beta
\end{equation}
has a unique solution.
\end{theorem}

\begin{proof}
A twice continuously differentiable function $\eta$ from a neighborhood of $[a, b]$ into $\R$ solves the ordinary 
differential equation in \eqref{orig_bvp.1} if and only the mapping $(y(t), u(t)) = (\eta(t), \eta'(t))$ from that
neighborhood into $\R^2$ solves the integral equation
\[
\begin{pmatrix}
y(t) \\
u(t)
\end{pmatrix}
=
\begin{pmatrix}
y(a) \\
u(a)
\end{pmatrix}
+
\int_a^t
\begin{pmatrix}
u(s) \\
f(s, y(s), u(s)) 
\end{pmatrix}
\,
ds.
\]
By the discussion surrounding \eqref{e:gamma} $\eta$ meets both boundary conditions in \eqref{orig_bvp.1} if and only if 
$\eta(a) = \alpha$ and $\eta'(a) = \gamma$ where $\gamma$ is given by \eqref{e:gamma}, with $(y(s), u(s))$ replaced by 
$(\eta(s), \eta'(s))$. In short the boundary value problem \eqref{orig_bvp.1} is equivalent to the integral equation
\begin{equation}\label{long.ie}
\begin{pmatrix}
y(t) \\
u(t)
\end{pmatrix}
=
\begin{pmatrix}
\alpha \\
\tfrac{1}{b - a} \left[ \beta -\alpha - \int_a^b (b - s) f(s, y(s), u(s)) \, ds \right]
\end{pmatrix}
+
\int_a^t
\begin{pmatrix}
u(s) \\
f(s, y(s), u(s))
\end{pmatrix}
ds.
\end{equation}
If $\yb(t) = (y(t), u(t))$ is a bounded continuous mapping from a neighborhood $U$ of $[a, b]$ in $\R$ into $\R^2$
then the right hand side of \eqref{long.ie} is well defined and defines a bounded continuous mapping from $U$ into
$\R^2$. Thus letting $\cc$ denote the set of bounded continuous mappings from a fixed bounded open neighborhood $U$ 
of $[a, b]$ into $\R^2$, a twice continuously differentiable function $\eta$ on $U$ into $\R$ solves the boundary
value problem \eqref{orig_bvp.1} if and only if $\eb \stackrel{\textrm{def}}{=} (\eta, \eta')$ is a fixed point
of the operator $\T : \cc \to \cc$ defined by
\renewcommand{\baselinestretch}{.7}
\[
\T
\begin{pmatrix}
y \\
u
\end{pmatrix}
(t)
=
\begin{pmatrix}
\alpha \\
\left[ \tfrac{\beta -\alpha}{b - a} - \int_a^b \tfrac{b - s}{b - a} f(s, y(s), u(s)) \, ds \right]
\end{pmatrix}
+
\int_a^t
\begin{pmatrix}
u(s) \\
f(s, y(s), u(s))
\end{pmatrix}
ds,
\]
\renewcommand{\baselinestretch}{2}
which we abbreviate to
\renewcommand{\stretch}{.7}
\begin{equation}\label{Cont}
\begin{aligned}
\T(\yb)(t)
=
\begin{pmatrix}
\alpha \\
 \left[ \tfrac{\beta -\alpha}{b - a} - \int_a^b \tfrac{b - s}{b - a} f(s, \yb(s)) \, ds \right]
\end{pmatrix}
+
\int_a^t
\Fb(s, \yb(s)) \,ds
\end{aligned}
\end{equation}
\renewcommand{\baselinestretch}{2}
by defining $\Fb : [a, b] \times \R^2 \to \R$ by
$\Fb(t, y, u) = (u, f(t, y, u))$.

The vector space $\cc$ equipped with the supremum norm is well known to be complete. Thus by the Contraction Mapping 
Theorem the theorem will be proved if we can show that $\T$ is a contraction on $\cc$. To this end, let $\eb$ and 
$\mb$ be elements of $\cc$. Let $\varepsilon = \max \{ t - b : t \in U \}$. Then for any $t \in U$
\renewcommand{\stretch}{.7}
\begin{align*}
| (\T&\eb)(t) - (\T\mb)(t) |_\textrm{sum} \\
&\leqslant 
\left| 
\begin{pmatrix} 
\alpha \\
\left[ \frac{\beta -\alpha}{b - a}  - \int_a^b \frac{b - s}{b - a} f(s, \eb(s)) ds \right] 
\end{pmatrix} 
-
\begin{pmatrix} 
\alpha \\
\left[ \frac{\beta -\alpha}{b - a} - \int_a^b \frac{b - s}{b - a} f(s, \mb(s)) ds \right] 
\end{pmatrix} 
\right|_\textrm{sum} \\
&\mspace{200mu}+
\left|
\int_a^t \Fb(s, \eb(s)) - \Fb(s, \mb(s)) \, ds
\right|_\textrm{sum} \\
&\leqslant  \int_a^b \frac{b - s}{b - a} | f(s, \eb(s)) - f(s, \mb(s)) | \, ds 
                 + \int_a^t | \Fb(s, \eb(s)) - \Fb(s, \mb(s)) |_\textrm{sum} \, ds \\
&\overset{(*)}{\leqslant} \int_a^b \frac{b - s}{b - a} L | \eb(s) - \mb(s) |_\textrm{sum} \, ds 
                 + \int_a^t (1 + L) | \eb(s) - \mb(s) |_\textrm{sum} \, ds \\
&\leqslant \int_a^b \frac{b - s}{b - a} L || \eb - \mb ||_\textrm{sup} \, ds
                 + \int_a^t (1 + L) || \eb - \mb ||_\textrm{sup} \, ds \\
&\leqslant [ \tfrac12 (b - a) L + (1 + L) ((b - a) + \varepsilon ) ] || \eb - \mb ||_\textrm{sup}
\end{align*}
where for inequality ($*$) Lemma \ref{lemma.L} was applied in the second summand. Thus 
$|| \T\eb - \T\mb ||_\textrm{sup} \leqslant (1 + \frac32 L) ((b - a) + \varepsilon)) || \eb - \mb ||_\textrm{sup}$ 
and $\T$ is a contraction provided $(1 + \frac32 L) ((b - a) + \varepsilon) < 1$, equivalently, provided 
$(1 + \frac32 L) (b - a) < 1 - (1 + \frac32 L)\varepsilon$. But $U$ can be chosen arbitrarily, hence 
$(1 + \frac32 L)\varepsilon$ can be made arbitrarily small, giving the sufficient condition of the theorem. 
\end{proof}

Repeated composition of the mapping $\T$ in the proof of the theorem generates Picard iterates. The recursion 
\eqref{e:iteration.split} could be expressed without reference to $\gamma$ and so as to exactly form the Picard iterates 
based on the contraction $\T$ simply by 
replacing $\gamma^{[k+1]}$ in the last line in \eqref{e:iteration.split.recursion} by the right 
hand side of the expression for $\gamma^{[k+1]}$ in the first line in \eqref{e:iteration.split.recursion}. The recursion 
\eqref{e:iteration.split} was expressed as it was so as to match the discussion leading up to it and to make it convenient to 
track the estimates of $\gamma$ in applications of the algorithm in Section \ref{s:exa}. Therefore, since by the Contraction 
Mapping Theorem iterates $\T^n(\eb) = (\T \circ \cdots \circ \T)(\eb)$ converge to the fixed point for every choice of starting 
point, the iterates defined by \eqref{e:iteration.split} will converge to the unique solution of \eqref{orig_bvp.1} that the 
theorem guarantees to exist. Thus we have the following result.

\begin{theorem} \label{r:alg1.works}
Let $f : [a, b] \times \R^2 \to \R : (t, y,u) \mapsto f(t, y, u)$ be Lipschitz in $\yb = (y, u)$ with Lipschitz
constant $L$ with respect to absolute value on $\R$ and the sum norm on $\R^2$. If 
$0 < b - a < (1 + \frac32 L)^{-1}$ then for any $\alpha$, $\beta \in \R$ the iterates generated by
Algorithm \ref{r:alg1} converge to the unique solution of the boundary value problem
\begin{equation} \label{orig_bvp}
y'' = f(t, y, y'), \qquad y(a) = \alpha, \quad y(b) = \beta
\end{equation}
guaranteed by Theorem \ref{thm.1} to exist.
\end{theorem}

\begin{remark}\label{mono.rmk}
Because Theorem \ref{thm.1} was proved by means of the Contraction Mapping Theorem approximate solutions 
$(y^{[k]}(t), u^{[k]}(t))$ to \eqref{long.ie} converge monotonically and exponentially fast. Since the sum norm was used 
on $\R^2$, however, the corresponding approximate solutions $y^{[k]}(t)$ to \eqref{orig_bvp.1} need not converge 
monotonically to the solution $y(t)$ of \eqref{orig_bvp}.
\end{remark}

\begin{remark}\label{mixed.rmk}
The ideas developed in this section can also be applied to two-point boundary value problems of the form 
\[
y'' = f(t, y, y'), \qquad y'(a) = \gamma, \quad y(b) = \beta \, ,
\]
so that in equation \eqref{basic.ivp} the constant $\gamma$ is now known and $\alpha = y(a)$ is unknown. Thus in 
\eqref{equiv.ivp} we evaluate at $t = b$ but now solve for $\alpha$ instead of $\gamma$, obtaining in place
of \eqref{e:gamma} the expression
\[
\alpha = \beta - \gamma (b - a) - \int_a^b (b - s) f(s, y(s), y'(s)) \, ds.
\]
In the Picard iteration scheme we now successively update an approximation of $\alpha$ starting with some initial
value $\alpha_0$. Theorems analogous to Theorems \ref{thm.1} and \ref{r:alg1.works} hold in this setting.
\end{remark}

The examples in Section \ref{s:exa} will show that the use of Algorithm \ref{r:alg1} is by no means restricted
to problems for which the hypotheses of Theorem \ref{r:alg1.works} are satisfied. It will in fact give satisfactory
results for many problems that do not satisfy those hypotheses.

\section{The Computational Method}\label{s:aux}

When the right hand side of the equation \eqref{orig_bvp} is a polynomial function then the integrations that
are involved in implementing Algorithm \ref{r:alg1} can always be done efficiently, but otherwise the Picard iterates 
can lead to impossible integrations. Consider, for example, the boundary value problem
\begin{equation}\label{psm1}
y'' = \sin y, \quad y(0) = 0, \quad y(\pi/8) = 1,
\end{equation}
with the corresponding first order system (with unknown constant $\gamma$)
\begin{equation}\label{psm1.sys}
\begin{aligned}
y' &= u      \\
u' &= \sin y
\end{aligned}
\qquad
\begin{aligned}
y(0) &= 0      \\
u(0) &= \gamma
\end{aligned}
\end{equation}
for which \eqref{e:iteration.split.recursion}, with explicit mention of $\gamma^{[k+1]}$ eliminated, is
\begin{equation}\label{sine.recursion}
\begin{aligned}
y^{[k+1]}(t) &= \int_0^t u^{[k]}(s) \, ds \\
u^{[k+1]}(t) &= \tfrac{8}{\pi}\Big[ 1 - \int_0^{\frac{\pi}{8}} (\tfrac{\pi}{8} - s) \sin y^{[k]}(s) \, ds \Big]
                + \int_0^t \sin y^{[k]}(s) \, ds.
\end{aligned}
\end{equation}
The first few iterates are readily computed but on the fourth iteration the expression for $u^{[4]}(t)$ contains terms
like $\int_0^t \sin(\frac{\pi^2}{64} \sin(\frac{8}{\pi}s)) \, ds$, which cannot be computed in closed form. In such a 
situation we use the auxiliary variable method as expounded by Parker and Sochacki (\cite{PS}; see also \cite{CPSW}). 
In this example we introduce the variable $v = \sin y$ and, since $v' = - \cos y \, y'$, the variable $w = \cos y$, so 
that \eqref{psm1.sys} is replaced by the four-dimensional problem
\begin{equation}\label{e:psm1.ult}
\begin{aligned}
y' &=   u  \\
u' &=   v  \\
v' &=  uw  \\ 
w' &= -uv
\end{aligned}
\qquad
\begin{aligned}
y(0) &= 0      \\
u(0) &= \gamma \\    
v(0) &= 0      \\
w(0) &= 1 \,,
\end{aligned}
\end{equation}
where the initial values for $v$ and $w$ come from their definitions in terms of $y(t)$ and the initial value of $y$.
Suppose that the unique solution to \eqref{psm1.sys} is $(y, u) = (\sigma(t), \tau(t))$ on some interval $J$ about 0 and that the 
unique solution to \eqref{e:psm1.ult} is $(y, u, v, w) = (\rho(t), \mu(t), \nu(t), \xi(t))$ on some interval $K$ about 0. Then by
construction $(y, u, v, w) = (\sigma(t), \tau(t), \sin \sigma(t), \cos \sigma(t))$ solves \eqref{e:psm1.ult} on $J$, hence we
conclude that on $J \cap K$ the function $y = \sigma(t)$, which in the general case we cannot find explicitly, is equal to the 
function $y = \rho(t)$, which we can approximate on any finite interval about 0 to any required accuracy. For although the 
dimension has increased, now the right hand sides of the differential equations are all polynomial functions so quadratures can 
be done easily. 

Applying the method of auxiliary variables in the implementation of Algorithm \ref{r:alg1} in general means simply adjoining to
\eqref{e:iteration.split.initial} initializations of auxiliary variables, say by their initial values as determined by their
definitions and the initial values $y(a) = \alpha$ and $u(a) = \gamma$, and adjoining to \eqref{e:iteration.split.recursion} the
obvious Picard recurrence expression arising from the initial value problems for the auxiliary variables, analogous to the last
two equations in \eqref{e:psm1.ult}. For example, for \eqref{psm1.sys} we obtain from \eqref{e:psm1.ult} the additional
initializations $v^{[0]}(t) \equiv 0$ and $w^{[0]}(t) \equiv 1$ and the additional recursion equations
$v^{[k+1]}(t) = \int_0^t u^{[k]}(s) w^{[k]}(s) \, ds$ and $w^{[k+1]}(t) = 1 - \int_0^t u^{[k]}(s) v^{[k]}(s) \, ds$.
Thus in general we obtain a computationally efficient method for approximating the solution.

The following theorem is stated and proved in \cite{PS}.

\begin{theorem}\label{r:PSM}
Let $\F = (f_1, \cdots, f_n) : \R^n \to \R^n$ be a polynomial mapping and $\yb = (y_1, \cdots, y_n) : \R \to \R^n$.  
Consider initial value problem  
\[
y'_j = f_k(\yb), \qquad y_j(0) = \alpha_j, \quad j = 1, \cdots, n
\]
and the corresponding Picard iterates $P_k(s) = (P_{1,k}(s), \cdots, P_{n,k}(s))$,
\begin{align*} 
P_{j,1}  (t) &= \alpha_j, \quad j = 1, \cdots, n \\
P_{j,k+1}(t) &= \alpha_j + \int_0^t f_j(P_k(s)) \, ds, \quad k = 1, 2, \cdots, \quad j = 1, \cdots, n.
\end{align*}
Then $P_{j,k+1}$ is the $k^{th}$ Maclaurin Polynomial for $y_j$ plus a polynomial all of whose terms have degree 
greater than k.
\end{theorem}

In \cite{CPSW} the authors address the issue as to which systems of ordinary differential equations can be handled by
this method. The procedure for defining the new variables is neither algorithmic nor unique. However, with sufficient
ingenuity it has been successfully applied in every case for which the original differential equation is analytic.

By Theorem \ref{r:PSM} Algorithm \ref{r:alg1} is generating approximations of the Maclaurin series of the solution of 
\eqref{basic.bvp}, but because of the convergence to $\gamma$ the approximations match the solution virtually perfectly 
at the right endpoint as well as at the left.

Because the method involves only repeated integration of polynomial functions, it is easy to code and experience
shows that it compares favorably in the computational time required to obtain results with accuracy comparable to
that obtained by means of such popular methods as the shooting method with fourth order Runge-Kutta numerics, the
power series method, and the finite difference method.

\section{Examples}\label{s:exa}

We will illustrate the method and its efficiency with several examples, both linear and nonlinear, and compare the 
result with the known exact solutions. Computations were done independently using Mathematica 10 and Maple 16.

\begin{example}\label{ex:-y.a}
Consider the second order linear ordinary differential equation $y'' = -y$. Since the right hand side is a polynomial 
function and is Lipschitz with Lipschitz constant 1 Algorithm \ref{r:alg1} can be applied directly to any corresponding 
two-point boundary value problem, and Theorem \ref{r:alg1.works} guarantees that the approximations to the solution of
the equivalent problem of the form \eqref{long.ie} that are generated will converge to the solution monotonically and 
exponentially fast, with respect to the supremum norm, on any interval $[a, b]$ for which $b - a < 2/5$. We will consider
two problems for which the solution is  is $y(t) = \cos t + \sin t$. The first is
\[
y'' = -y, \qquad y(0) = 1, \quad y(\tfrac\pi8) = \sqrt{1 + 1 / \sqrt{2}},
\]
for which the conditions of Theorem \ref{r:alg1.works} are met, and for which we find that for the first five iterates 
the errors, $|| y(t) - y^{[k]}(t) ||_\text{sup}$, $1 \leqslant k \leqslant 5$, rounded to five decimal places, are
\[
0.022 60 \qquad 0.003 39 \qquad 0.000 36 \qquad 0.000 05 \qquad 0.000 01
\]
and that $| \gamma - \gamma^{[k]} |$ for $1 \leqslant k \leqslant 5$ are
\[
0.022 94 \qquad 0.002 93 \qquad 0.000 41 \qquad 0.000 04 \qquad 0.000 01.
\]

However, Algorithm \ref{r:alg1} performs well even for intervals of length much greater than $2/5$. This is the case for 
the second problem (with the same solution),
\begin{equation}\label{e:ex.-y.a}
y'' = -y, \qquad y(0) = 1, \quad y(\tfrac\pi4) = \sqrt{2}.
\end{equation}
When Algorithm \ref{r:alg1} is applied to \eqref{e:ex.-y.a} we have that $|| y(t) - y^{[k]}(t) ||_\text{sup}$ for 
$1 \leqslant k \leqslant 5$ are
\[
0.100 00 \qquad 0.002 30 \qquad 0.006 40 \qquad 0.001 42 \qquad 0.000 40
\]
(see Remark \ref{mono.rmk}) and that $| \gamma - \gamma^{[k]} |$ for $1 \leqslant k \leqslant 5$ are
\[
0.079 91 \qquad 0.025 69 \qquad 0.005 50 \qquad 0.001 50 \qquad 0.000 35.
\]
As already noted, because we are forcing agreement at the two endpoints the error is virtually zero at each end of
the interval, which is the universal pattern.
\end{example}

\begin{example} \label{Nex1}
Consider non-linear boundary value problem
\begin{equation} \label{e:nonlin.bvp.example}
y'' = 16 + (3 - 2t)^3 + \tfrac14 y y', \qquad y(0) = \tfrac{43}{3}, \quad y(1) = 17,
\end{equation}
which has solution $y(t) = (3 - 2t)^2 + 16 (3 - 2t)^{-1}$. The right hand side of the differential equation is not Lipschitz in 
$\yb = (y, y')$ and the right endpoint 1 is close to where the solution blows up. Nevertheless when we attempt to apply Algorithm 
\ref{r:alg1} the iterates converge to the exact solution, albeit slowly. After eight iterations we have that 
$|| y(t) - y^{[8]}(t) ||_\text{sup} \approx 0.024$ and $\gamma^{[8]} \approx -8.457$ compared to $\gamma = -8.\bar4$. 
\end{example}

\begin{example}
Consider the boundary value problem
\begin{equation}\label{NCVRD.1}
y'' = -e^{-2y}, \qquad y(0) = 0, \quad y(1.2) = \ln \cos 1.2 \approx  -1.015\,123\,283,
\end{equation}
for which auxiliary variables must be introduced. The right hand side is not Lipschitz in $y$ yet in this case the algorithm 
works well. The unique solution is $y(t) = \ln \cos t$, yielding $\gamma = 0$.

Introducing the dependent variable $u = y'$ to obtain the equivalent first order system $y' = u$, $u' = e^{-y}$ and the 
variable $v = e^{-2y}$ to replace the transcendental function with a polynomial we obtain the expanded system
\begin{align*}
y' &=  u   \\
u' &= -v   \\
v' &= -2uv
\end{align*}
with initial conditions
\[
y(0) = 0, \quad u(0) = \gamma, \quad v(0) = 1
\]
(with $\gamma$ regarded as unknown here), a system on $\R^3$ for which the $y$-component is the solution of the boundary value problem
\eqref{NCVRD.1}. Thus in this instance
\[
y^{[0]}(t)   \equiv 0,                        \quad
u^{[0]}(t)   \equiv \frac{\ln \cos 1.2}{1.2}, \quad
v^{[0]}(t)   \equiv 1,                        \quad
\]
and
\begin{align*}
\gamma^{[k+1]} &= \Big( \ln \cos 1.2 + \int_0^{1.2} \, (1.2 - s) v^{[k]}(s) \, ds \Big) / 1.2 \\
y^{[k+1]}(t) &=       0        +   \int_0^t u^{[k]}(s)            \, ds \\
u^{[k+1]}(t) &= \gamma^{[k+1]} -   \int_0^t v^{[k]}(s)            \, ds \\
v^{[k+1]}(t) &=       1        - 2 \int_0^t u^{[k]}(s) v^{[k]}(s) \, ds \, .
\end{align*}
The first eight iterates of $\gamma$ are:
\[
\gamma^{[1]} = -0.24594,
\mspace{15mu}
\gamma^{[2]} =  \phantom{-}0.16011,
\mspace{15mu}
\gamma^{[3]} =  \phantom{-}0.19297,
\mspace{15mu}
\gamma^{[4]} =  0.04165,
\]
\[
\gamma^{[5]} = -0.04272,
\mspace{15mu}
\gamma^{[6]} = -0.04012,
\mspace{15mu}
\gamma^{[7]} = -0.00923,
\mspace{15mu}
\gamma^{[8]} =  0.01030,
\]
The maximum errors show a similar sort of pattern as they tend to zero; after eight iterations the maximum error is 
$|| y(t) - y^{[8]}(t) ||_\text{sup} \approx 0.0115$.
\end{example}

\section{Extended Theorem and Algorithm}\label{s:extended}

Theorem \ref{thm.1}, hence the theoretical scope of the algorithm presented in Section \ref{s:alg}, can be extended 
by partitioning the interval $[a, b]$ into $n$ subintervals and simultaneously and recursively approximating the 
solutions to the $n$ boundary value problems that are induced on the subintervals by \eqref{basic.bvp} and its 
solution.

If $\eta(t)$ is a solution of the original boundary value problem \eqref{basic.bvp} and the interval $[a, b]$ is 
subdivided into $n$ subintervals of equal length $h$ by means of a partition
\[
a = t_0 < t_1 < t_2 < \cdots < t_n = b
\]
then setting $\beta_j = \eta(t_j)$, $j = 1, \ldots, n-1$, we see that $n$ boundary value problems are induced:
\[
\begin{aligned}
y'' &= f(t, y, y') \\
y(t_0) &= \alpha, \mspace{5mu} y(t_1) = \beta_1
\end{aligned}
\quad
\begin{aligned}
y'' &= f(t, y, y') \\
y(t_1) &= \beta_1, \mspace{5mu} y(t_2) = \beta_2
\end{aligned}
\quad
\dots
\quad
\begin{aligned}
y'' &= f(t, y, y') \\
y(t_{n-1}) &= \beta_{n-1}, \mspace{5mu} y(t_n) = \beta.
\end{aligned}
\]
Setting $\gamma_j = \eta'(t_{j-1})$, $j = 1, \dots, n$, or computing them by means of an appropriate implementation 
of \eqref{e:gamma}, their solutions are solutions of the respective initial value problems
\[
\begin{aligned}
y''  = f&(t, y, y') \\
y(t_0)  &= \alpha,  \\
y'(t_0) &= \gamma_1
\end{aligned}
\mspace{40mu}
\begin{aligned}
y'' =  f&(t, y, y') \\
y(t_1)  &= \beta_1, \\
y'(t_1) &= \gamma_2
\end{aligned}
\mspace{40mu}
\dots
\mspace{40mu}
\begin{aligned}
y''      = f&(t, y, y')         \\
y(t_{n-1})  &= \beta_{n-1}, \\
y'(t_{n-1}) &= \gamma_n.
\end{aligned}
\]
We denote the solutions to these problems by $y_j(t)$ with derivatives $u_j(t) := y_j'(t)$, $j = 1, 2, 3, \dots, n$.
To make the presentation cleaner and easier to read we will use the following shorthand notation (for relevant choices 
of $j$), where the superscript $[k]$ will pertain to the $k$th iterate in the recursion to be described:
\begin{equation}\label{e:short.notation}
\begin{gathered}
f_j(s) = f(s, y_j(s), u_j(s))
\qquad
f_j^{[k]}(s) = f(s, y_j^{[k]}(s), u_j^{[k]}(s)) \\
\phantom{blank line} \\
\yb_j(s) = \begin{pmatrix} u_j(s) \\ y_j(s) \end{pmatrix}
\qquad
\yb_j^{[k]}(s) = \begin{pmatrix} u_j^{[k]}(s) \\ y_j^{[k]}(s) \end{pmatrix} \\
\phantom{blank line} \\
I_j = \int_{t_{j-1}}^{t_j} f_j(s) \, ds
\qquad
I_j^{[k]} = \int_{t_{j-1}}^{t_j} f_j^{[k]}(s) \, ds \\
\phantom{blank line} \\
J_j = \int_{t_{j-1}}^{t_j} (t_j - s) f_j(s) \, ds
\qquad
J_j^{[k]} = \int_{t_{j-1}}^{t_j} (t_j - s) f_j^{[k]}(s) \, ds.
\end{gathered}
\end{equation}

The idea for generating a sequence of successive approximations of $\eta(t)$ is to update the estimates of the 
functions $\yb_j(t)$ using 
\begin{equation}\label{e:ur.update.fns}
\yb_j(t) = \begin{pmatrix} \beta_{j-1} \\ \gamma_j \end{pmatrix}
         + \int_{t_{j-1}}^t \begin{pmatrix} u_j(s) \\ f_j(s) \end{pmatrix} \, ds
\end{equation}
and then update $\gamma_j$ and $\beta_j$ using 
\begin{equation}\label{e:ur.updates}
\gamma_j = \gamma_{j-1} + I_{j-1}
\qquad
\text{and}
\qquad
\beta_{j-1} =  \beta_j - h \, \gamma_j - J_j
\end{equation}
(starting with $j = 1$ and working our way up to $j = n$ for the $\gamma_j$ and in the reverse order with the 
$\beta_j$, with the convention that $\beta_n = \beta$), except that on the first step we update $\gamma_1$ using 
instead 
\begin{equation}\label{e:ur.update.gamma1}
\gamma_1 = \tfrac1h [ \beta_1 - \alpha - J_1 ]
\end{equation}
and there is no $\beta_0$. Note that the updates on the $\beta_j$ come from ``the right," i.e., values of 
$\beta_r$ with $r > j$, hence ultimately tying into $\beta$ at each pass through the recursion, while the updates 
on the $\gamma_j$ come from ``the left," i.e., values of $\gamma_r$ with $r < j$, hence ultimately tying into 
$\alpha$ at each pass through the recursion.

In fact we will not be able to make the estimates that we need to show convergence if we update $\beta_1$ on the 
basis given above. To obtain a useful formula on which to base the successive approximations of $\beta_1$, we
begin by using the second formula in \eqref{e:ur.updates} $n-1$ times:
\begin{align*}
\beta_1 &= \beta_2 - (h \, \gamma_2 + J_2) \\
        &= \beta_3 - (h \, \gamma_3 + J_3) - (h \, \gamma_2 + J_2) \\
        &= \beta_4 - (h \, \gamma_4 + J_4) - (h \, \gamma_3 + J_3) - (h \, \gamma_2 + J_2) \\
        &\mspace{20mu}\vdots \\
        &= \beta - (h \, \gamma_n + J_n) - \cdots - (h \, \gamma_2 + J_2) \\
        &= \beta - h(\gamma_2 + \cdots + \gamma_n) - (J_2 + \cdots + J_n).
\end{align*}
But by repeated application of the first equation in \eqref{e:ur.updates} and use of
\eqref{e:ur.update.gamma1} on the last step
\begin{align*}
&\mspace{20mu}\gamma_2 + \cdots + \phantom{3}  \gamma_{n-2} + \phantom{2} \gamma_{n-1} + \gamma_n \\
&= \gamma_2  + \cdots + \phantom{3} \gamma_{n-2} + 2 \gamma_{n-1} + I_{n-1}  \\
&= \gamma_2  + \cdots + 3 \gamma_{n-2} + 2 I_{n-2}      + I_{n-1}  \\
&\mspace{20mu}\vdots \\
&=(n-1)\gamma_2 + (n-2) I_2 + \cdots + 3 I_{n-3} + 2 I_{n-2} + I_{n-1} \\
&=(n-1) \gamma_1 + (n-1) I_1 + (n-2) I_2 + \cdots + 3 I_{n-3} + 2 I_{n-2} + I_{n-1} \\
&= \frac{n-1}{h} [ \beta_1 - \alpha - J_1 ]
 + (n-1) I_1 + (n-2) I_2 + \cdots + 3 I_{n-3} + 2 I_{n-2} + I_{n-1}.
\end{align*}
Inserting this expression into the previous display and solving the resulting equation for $\beta_1$ yields the 
formula
\begin{equation}\label{e:ur.update.beta1}
\beta_1 = \frac1n
          \left[
          \beta + (n-1) \alpha + (n-1) J_1 - \sum_{r=2}^n J_r - h \sum_{r=1}^{n-1} (n-r) I_r
          \right].
\end{equation}
Once an initialization has been chosen, an iteration procedure based on \eqref{e:ur.update.fns}, 
\eqref{e:ur.updates}, \eqref{e:ur.update.gamma1}, and \eqref{e:ur.update.beta1} is, with the convention 
$\beta_0 = \alpha$ and $\beta_n = \beta$, the shorthand notation introduced above, and order of evaluation in the
order listed,
\begin{subequations}\label{e:update.n}
\begin{gather}
\yb_j^{[k+1]}(t) = \begin{pmatrix} \beta_{j-1}^{[k]} \\ \gamma_j^{[k]} \end{pmatrix}
         + \int_{t_{j-1}}^t \begin{pmatrix} u_j^{[k]}(s) \\ f_j^{[k]}(s) \end{pmatrix} \, ds
\qquad j = 1, \cdots, n
\label{e:update.n.fns} \\
\phantom{blank line} \notag \\
\beta_1^{[k+1]} = \frac1n
                  \left[
                  \beta + (n-1) \alpha + (n-1) J_1^{[k+1]}
                  - \sum_{r=2}^n J_r^{[k+1]} - h \sum_{r=1}^{n-1} (n-r) I_r^{[k+1]}
                  \right]
\label{e:update.n.beta1} \\
\phantom{blank line} \notag \\
\gamma_1^{[k+1]} = \tfrac1h [ \beta_1^{[k+1]} - \alpha -J_1^{[k+1]} ]
\label{e:update.n.gamma1} \\
\phantom{blank line} \notag \\
\gamma_j^{[k+1]} = \gamma_{j-1}^{[k+1]} + I_{j-1}^{[k+1]}
\qquad j = 2, \cdots, n
\label{e:update.n.gamma.j}
\\
\phantom{blank line} \notag \\
\beta_{j-1}^{[k+1]} = \beta_j^{[k+1]} - h \, \gamma_j^{[k+1]} - J_j^{[k+1]}
\qquad
j = n, n-1, \ldots, 4, 3
\label{e:update.n.beta.j}
\end{gather}
\end{subequations}

\begin{theorem}\label{r:Thm.Mul}
Suppose the function $f(t, y, u)$ from $[a, b] \times \R^2$ into $\R$ is continuous and Lipschitz in $\yb = (y,u)$ 
with Lipschitz constant $L$. If there exists an integer $n \geqslant 1$ such that for the subdivision of $[a, b]$ 
into $n$ subintervals of equal length $h$ by the partition
\[
a = t_0 < t_1 < \cdots < t_{n-1} < t_n = b
\]
the inequality
\[
\frac1{2n} [ (n^3 + n^2 + n + 2) L + 2 ] (b - a) < 1
\]
holds if $h = (b-a)/n \leqslant 1$ or the inequality
\[
\frac1{2n^2} [ (n^3 + n^2 + n + 2) L + 2 ] (b - a)^2 < 1
\]
holds if $h = (b-a)/n \geqslant 1$, then there exists a solution of the two-point boundary value problem 
\eqref{basic.bvp}. Moreover, in the language of the notation introduced in the first paragraph of this section and 
display \eqref{e:short.notation}, for any initial choice of the functions $\yb_j(t) = (y_j(t), u_j(t))$, 
$1 \leqslant j \leqslant n$, the constants $\gamma_j$, $1 \leqslant j \leqslant n$, and the constants $\beta_j$, 
$1 \leqslant j \leqslant n-1$, the sequence of successive approximations defined by \eqref{e:update.n} converges to 
such a solution.
\end{theorem}

The following three lemmas will be needed in the proof. The straightforward proofs are omitted.

\begin{lemma}\label{r:cauchy.cond}
Let $x^{[k]}$ be a sequence in a normed vector space $(V, | \cdot |)$. If there exist a number $c < 1$ and an index 
$N \in \N$ such that
\[
| x^{[k+1]} - x^{[k]} | \leqslant c | x^{[k]} - x^{[k-1]} |
\quad
\text{for all} \quad k \geqslant N
\]
then the sequence $x^{[k]}$ is a Cauchy sequence.
\end{lemma}

\begin{lemma}\label{r:fyjk.estimate}
Suppose the interval $[a, b]$ has been partitioned into $n$ subintervals of equal length $h = (b - a) / n$ by 
partition points $a = t_0 < t_1 < \cdots < t_{n-1} < t_n = b$. With the notation
\[
\yb_j(t) = \begin{pmatrix} y_j(t) \\ u_j(t) \end{pmatrix}
\quad
\text{and}
\quad
f_j^{[r]}(s) = f(s, y_j^{[r]}(s), u_j^{[r]}(s)),
\quad
r \in \Z^+ \cup \{ 0 \},
\quad
j = 1, 2,
\]
the following estimates hold:
\begin{equation}\label{int.est.simple}
\int_{t_{j-1}}^{t_j} | f_j^{[k+1]}(s) - f_j^{[k]}(s) | \, ds
\leqslant
L \, h || \yb^{[k+1]} - \yb^{[k]} ||_\textrm{max}
\end{equation}
and
\begin{equation}\label{int.est.complex}
\int_{t_{j-1}}^{t_j} (t_j - s) | f_j^{[k+1]}(s) - f_j^{[k]}(s) | \, ds
\leqslant
\tfrac 12 \, L \, h^2 || \yb^{[k+1]} - \yb^{[k]} ||_\textrm{max}.
\end{equation}
\end{lemma}

\begin{lemma}\label{r:c2.lemma}
Suppose $\eta : (-\epsilon, \epsilon) \to \R$ is continuous and that $\eta'$ exists on 
$(-\epsilon, 0) \cup (0, \epsilon)$. Suppose $g : (-\epsilon, \epsilon) \to \R$ is continuous and that $g = \eta'$ 
on $(-\epsilon, 0) \cup (0, \epsilon)$. Then $\eta'$ exists and is continuous on $(-\epsilon, \epsilon)$.
\end{lemma}

\begin{proof}[Proof of Theorem \ref{r:Thm.Mul}]
We will show that the sequence
\[
\yb^{[k]}(t) = (\yb_1^{[k]}(t), \dots, \yb_n^{[k]}(t))
\in
C([t_0, t_1], \R^2) \times C([t_1, t_2], \R^2) \times \cdots \times C([t_{n-1}, t_n], \R^2)
\]
is a Cauchy sequence by means of Lemma \ref{r:cauchy.cond}, where we place the supremum norm on each function space, with 
respect to absolute value on $\R$ and the sum norm on $\R^2$, and the maximum norm on their product. It is clear that the 
successive approximations converge to functions on the individual subintervals which when concatenated form a function $y(t)$
that is $C^1$ on $[a, b]$, $C^2$ on $[a, b] \setminus \{ t_1, \dots, t_{n-1} \}$, solves the differential equation in 
\eqref{basic.bvp} on the latter set, and satisfies the two boundary conditions in \eqref{basic.bvp}. An application of 
Lemma \ref{r:c2.lemma} at each of the $n - 1$ partition points implies the existence of the second derivative at the partition 
points, so that $y(t)$ solves the boundary value problem \eqref{basic.bvp}.

To begin the proof that $\yb^{[k]}(t)$ is a Cauchy sequence, by Lemma \ref{r:fyjk.estimate}
\begin{equation}\label{int.est.simple.gen}
\begin{aligned}
| I_r^{[k+1]} - I_r^{[k]} |
&=
\left|
\int_{t_{r-1}}^{t_r}  f_r^{[k+1]}(s) \, ds - \int_{t_{r-1}}^{t_r} f_r^{[k]}(s) \, ds
\right|
\\
&\leqslant
\int_{t_{r-1}}^{t_r}  | f_r^{[k+1]}(s) - f_r^{[k]}(s) | \, ds
\\
&\leqslant
L \, h || \yb^{[k+1]} - \yb^{[k]} ||_\textrm{max}
\end{aligned}
\end{equation}
and
\begin{equation}\label{int.est.complex.gen}
\begin{aligned}
| J_r^{[k+1]} - J_r^{[k]} |
&=
\left
| \int_{t_{r-1}}^{t_r} (t_r - s) f_r^{[k+1]}(s) \, ds - \int_{t_{r-1}}^{t_r} (t_r - s) f_r^{[k]}(s) \, ds
\right|
\\
&\leqslant
\int_{t_{r-1}}^{t_r} (t_r - s) | f_r^{[k+1]}(s) - f_r^{[k]}(s) | \, ds
\\
&\leqslant
\tfrac12 \, L \, h^2 || \yb^{[k+1]} - \yb^{[k]} ||_\textrm{max}.
\end{aligned}
\end{equation}

Then from \eqref{e:update.n.beta1} we have
\begin{align*}
| \beta_1^{[k+1]} - \beta_1^{[k]} |
&\leqslant
\frac{n-1}{n}
\left| J_1^{[k+1]} - J_1^{[k]} \right|
+   \sum_{r=2}^n           \left| J_r^{[k+1]} - J_r^{[k]} \right| \\
&\mspace{300mu}+ h \sum_{r=1}^{n-1} (n-r) \left| I_r^{[k+1]} - I_r^{[k]} \right|
\\
&\leqslant
\left[
\left(\frac{n-1}{n}\right) \frac12 h^2
+ \sum_{r=2}^n \frac12 h^2
+ h \sum_{r=1}^{n-1}(n-r) h
\right]
L || \yb^{[k+1]} - \yb^{[k]} ||_\textrm{max}
\\
&= \widehat B_1 h^2 L || \yb^{[k+1]} - \yb^{[k]} ||_\textrm{max},
\end{align*}
where $\widehat B_1 = \frac{n-1}2 \left[ \frac1n + 1 + n \right]$.

From \eqref{e:update.n.gamma1}
\begin{align*}
| \gamma_1^{[k+1]} - \gamma_1^{[k]} |
&\leqslant
\frac1h | \beta_1^{[k+1]} - \beta_1^{[k]} | + \frac1h | J_1^{[k+1]} - J_1^{[k]} | \\
&\leqslant
[\widehat B_1 h L + \frac12 h L ] || \yb^{[k+1]} - \yb^{[k]} ||_\textrm{max} \\
&=
\widehat \Gamma_1 h L || \yb^{[k+1]} - \yb^{[k]} ||_\textrm{max}
\end{align*}
and from repeated application of \eqref{e:update.n.gamma.j}, starting from $j =2$ up through
$j = n$,
\begin{align*}
| \gamma_j^{[k+1]} - \gamma_j^{[k]} |
&\leqslant
| \gamma_{j-1}^{[k+1]} - \gamma_{j-1}^{[k]} | + | I_{j-1}^{[k+1]} - I_{j-1}^{[k]} | \\
&\leqslant
[\widehat \Gamma_{j-1} h L + h L ] || \yb^{[k+1]} - \yb^{[k]} ||_\textrm{max} \\
&=
\widehat \Gamma_j h L || \yb^{[k+1]} - \yb^{[k]} ||_\textrm{max}
\end{align*}
with $\widehat \Gamma_j = \widehat B_1 + \frac{2j-1}{2}$, which is in fact valid for $1 \leqslant j \leqslant n$.

From repeated application of \eqref{e:update.n.beta.j}, starting from $j = n-1$ (with the
convention that $\beta_n = \beta$) and down through $j = 2$,
\begin{align*}
| \beta_j^{[k+1]} - \beta_j^{[k]} |
&\leqslant
| \beta_{j+1}^{[k+1]} - \beta_{j+1}^{[k]} |
+ h | \gamma_{j+1}^{[k+1]} - \gamma_{j+1}^{[k]} |
+ | J_{j+1}^{[k+1]} - J_{j+1}^{[k]} | \\
&\leqslant
[ \widehat B_{j+1} h^2 L + \widehat \Gamma_{j+1} h^2 L + \frac12 h^2 L ]
                                    || \yb^{[k+1]} - \yb^{[k]} ||_\textrm{max} \\
&=
\widehat B_j h^2 L || \yb^{[k+1]} - \yb^{[k]} ||_\textrm{max},
\end{align*}
and $\widehat B_{n-r} = r \widehat B_1 + r \, n - \frac{r (r-1)}{2}$, hence
$\widehat B_j = (n - j) \widehat B_1 + n (n - j) - \frac{(n - j)(n - j - 1)}{2}$, $2 \leqslant j \leqslant n - 1$.

Using these estimates we find that, setting $h^* = \max \{ h, h^2 \}$, for any $j \in \{ 2, \dots, n \}$, for any
$t \in [t_{j-1}, t_j]$,
\begin{align*}
| \yb_j^{[k+1]}&(t) - \yb_j^{[k]}(t) |_\text{sum} \\
&\leqslant
\left|
\begin{pmatrix}
\beta_{j-1}^{[k]} - \beta_{j-1}^{[k-1]} \\
\gamma_j^{[k]} - \gamma_j^{[k-1]}
\end{pmatrix}
\right|_\text{sum}
+
\int_{t_{j-1}}^{t_j}
\left|
\begin{pmatrix} u_j^{[k]}(s) - u_j^{[k-1]}(s) \\
f_j^{[k]}(s) - f_j^{[k-1]}(s)
\end{pmatrix}
\right|_\text{sum} \, ds \\
&\leqslant
| \beta_{j-1}^{[k]} - \beta_{j-1}^{[k-1]} | + | \gamma_j^{[k]} - \gamma_j^{[k-1]} |
+
(1 + L) \int_{t_{j-1}}^{t_j} | \yb_j^{[k]}(s) - \yb_j^{[k-1]}(s) |_\text{sum} \, ds \\
&\leqslant
[ \widehat B_{j-1} h^2 L + \widehat \Gamma_j h L ] \, || \yb^{[k]} - \yb^{[k-1]} ||_\textrm{max}
+
(1 + L) || \yb_j^{[k]} - \yb_j^{[k-1]} ||_\textrm{sup} \, h \\
&\leqslant
[ \widehat B_{j-1} L + \widehat \Gamma_j L + (1 + L) ] h^* || \yb^{[k]} - \yb^{[k-1]} ||_\textrm{max} \\
&=
[ ( \widehat B_{j-1} + \widehat \Gamma_j + 1 ) L + 1 ] h^* || \yb^{[k]} - \yb^{[k-1]} ||_\textrm{max} \\
&= c_j h^* || \yb^{[k]} - \yb^{[k-1]} ||_\textrm{max}
\end{align*}
where, using the expressions above for $\widehat B_1$ and $\widehat \Gamma_1$,
\[
c_j = \tfrac1{2n} [ n^4 + (3 - j) n^3 + n^2 + (3 j - j^2)n + (j - 2)] L + 1 \qquad (2 \leqslant j \leqslant n).
\]
Similarly, for all $t \in [t_0, t_1]$,
$| \yb_1^{[k+1]}(t) - \yb_1^{[k]}(t) |_\text{sum} \leqslant c_1 h || \yb^{[k]} - \yb^{[k-1]} ||_\textrm{max}$
for
\[
c_1 = \left[  \frac{n^3 + 3n - 1}{2n} \right] L + 1.
\]
Then for all $j \in \{ 1, \dots, n \}$,
\[
|| \yb_j^{[k]} - \yb_j^{[k-1]} ||_\textrm{sup} \leqslant c_j h^* || \yb^{[k]} - \yb^{[k-1]} ||_\textrm{max}
\]
and
\begin{equation}\label{e:yb.estimate}
|| \yb^{[k+1]} - \yb^{[k]} ||_\textrm{max}
\leqslant
\max \{ c_1, \ldots, c_n \} h^* || \yb^{[k]} - \yb^{[k-1]} ||_\textrm{max}.
\end{equation}

For fixed $n$, for $j \geqslant 2$, $c_j$ is a quadratic function of $j$ with maximum at $j =\frac{-n^3 + 3 n + 1}{2n}$,
which is negative for $n \geqslant 2$, so $c_2 > c_j$ for $j \geqslant 3$. Direct comparison shows that $c_2 > c_1$ for
all choices of $n$ as well. Thus estimate \eqref{e:yb.estimate} is
\[
|| \yb^{[k+1]} - \yb^{[k]} ||_\textrm{max} \leqslant  c_2 h^* || \yb^{[k]} - \yb^{[k-1]} ||_\textrm{max}
\]
and by Lemma \ref{r:cauchy.cond} the sequence $\yb^{[k]}(t)$ is a Cauchy sequence provided
$c_2 h^* < 1$, which, when $h = (b-a)/n \leqslant 1$ is the condition
\[
\frac1{2n} [ (n^3 + n^2 + n + 2) L + 2 ] (b - a) < 1
\]
and when $h = (b-a)/n \geqslant 1$ is the condition
\[
\frac1{2n^2} [ (n^3 + n^2 + n + 2) L + 2 ] (b - a)^2 < 1.
\]
\end{proof}

To see that Theorem \ref{r:Thm.Mul} can provide an actual improvement over Theorem \ref{thm.1}, suppose the
interval $[a, b]$ is fixed and we wish to know how large $L$ can be and still be assured that a solution to \eqref{basic.bvp} exists.
For any $n > b - a$ the corresponding maximum value of $L$ allowed by Theorem \ref{r:Thm.Mul} is greater than that allowed by 
Theorem \ref{thm.1} if
\[
\tfrac23
\left[
(b - a)^{-1} - 1 
\right]
<
\frac{1}{n^3 + n^2 + n + 2} 
\left[
2n (b - a)^{-1} - 2
\right],
\]
which holds if and only if
\begin{equation}\label{e:L.est}
\frac{n^3 + n^2 - 2n+ 2}{n^3 + n^2 + n - 1} < b - a.
\end{equation}
For $n > 1$ the left hand side of \eqref{e:L.est} is less than 1 and increases with increasing $n$. Thus for example
if $b - a = 1$ then Theorem \ref{thm.1} guarantees that a solution to \eqref{basic.bvp} will exist if $L < 0$, so no
conclusion can be made, whereas for all $n \geqslant 2$ Theorem \ref{r:Thm.Mul} implies existence of a unique solution
if
\[
L < \frac{2n - 2}{n^3 + n^2 + n - 2}.
\]
The best result is for $n = 2$ and is $L < \frac16$.

Similarly, if a Lipschitz constant $L$ is known to exist for all values of $t$, $y$, and $y'$, then 
Theorem \ref{r:Thm.Mul} can sometimes provide a guarantee of existence of a solution to a boundary value problem of the
form \eqref{basic.bvp} on a longer interval than that provided by Theorem \ref{thm.1}.

\section{Conclusion}\label{s:concl}

We have introduced a purely symbolic technique for approximating solutions of two-point boundary value
problems whose output is a sequence of polynomials that converges to the true solution exponentially fast
with respect to the supremum norm. We provided conditions under which the method is guaranteed to work, and
illustrated by example that its practical usefulness exceeds what the theorems provide. By introducing auxiliary
variables we overcome problems with quadratures that cannot be performed in closed form. The algorithm is
easy to code in popular comptuer algebra systems such as Maple and Mathematica. Experience has shown that it
compares favorably in efficiency with shooting and finite difference approximation techniques.

\end{document}